\documentclass[10pt]{amsart}
\usepackage[dvips]{graphicx}
\usepackage{bbm}
\usepackage{epsfig}
\usepackage{caption}
\usepackage{color}
\usepackage{amsthm}
\usepackage{amsmath}
\usepackage{amsfonts}
\usepackage{amssymb}
\usepackage{graphics,graphicx}
\usepackage{grffile}

\usepackage{color}
\def\be{\begin{equation}}
\def\ee{\end{equation}}
\def\beq{\begin{eqnarray}}
\def\eeq{\end{eqnarray}}
\def\beqs{\begin{eqnarray*}}
\def\eeqs{\end{eqnarray*}}
\def\ea{\end{array}}
\def\ea{\end{array}}
\def\ds{\displaystyle}

\def\RR{\mathbb{R}}
\def\11{{\rm 1~\hspace{-1.5ex}1} }
\def\NN{\mathbb{N}}

\def\CC{\mathbb{C}}

\newcommand{\rfb}[1]{\mbox{\rm
   (\ref{#1})}\ifx\undefined\stillediting\else:\fbox{$#1$}\fi}

\makeatletter
\def\section{\@startsection {section}{1}{\z@}{-3.5ex plus -1ex minus
    -.2ex}{2.3ex plus .2ex}{\large\bf}}
\makeatother

\font\eufm=eufm10\font\eufms=eufm10\font\eufmss=eufm10\newfam\eufam
\textfont\eufam=\eufm\scriptfont\eufam=\eufms\scriptscriptfont\eufam=\eufmss

\usepackage{amsthm}
\swapnumbers
\newtheorem{theorem}{Theorem}[section]
\newtheorem{lemma}[theorem]{Lemma}

\newtheorem{remark}[theorem]{Remark}

\newtheorem{proposition}[theorem]{Proposition}

\begin{document}
\thispagestyle{empty}
\title[Dissipative Schr\"{o}dinger operator on the tadpole graph]{Spectral analysis and stabilization of the dissipative Schr\"{o}dinger operator on the tadpole graph}

\author{Ka\"{\i}s Ammari}
\address{UR Analysis and Control of PDEs, UR13E564, Department of Mathematics, Faculty of Sciences of
Monastir, University of Monastir, 5019 Monastir, Tunisia}
\email{kais.ammari@fsm.rnu.tn}

\author{Rachid Assel}
\address{UR Analysis and Control of PDEs, UR13E564, Department of Mathematics, Faculty of Sciences of
Monastir, University of Monastir, 5019 Monastir, Tunisia}
\email{rachid.assel@fsm.rnu.tn}

\date{}

\begin{abstract} 
We consider the damped Schr\"odinger semigroup $e^{-it \frac{d^2}{dx^2}}$
on the tadpole graph ${\mathcal R}$. We first give a careful spectral analysis and an appropriate decomposition of the kernel of the resolvent. As a consequence and by showing that the generalized eigenfunctions form a Riesz basis of some subspace of $L^2({\mathcal R})$, we prove that the corresponding energy decay exponentially. 
\end{abstract}


\maketitle

\tableofcontents
\vfill\break

\section{Introduction} \label{formulare}

In the last few years various physical models of multi-link flexible structures consisting of finitely or infinitely many interconnected flexible elements such as strings, beams, plates, shells have been of great interest. See the references \cite{amregmer} as well as \cite{nojaetall:15,cacc,kost} and the references therein. The spectral analysis of such structures has some applications to control or stabilization problems (cf. \cite{ammarinicaise}).

\medskip

In this paper we give a careful spectral analysis of the dissipative Schr\"odinger operator on the tadpole graph
(sometimes also called lasso graph) and as application we study the corresponding stabilization problem:

\begin{center}\label{fig1} 
\includegraphics[scale=0.80]{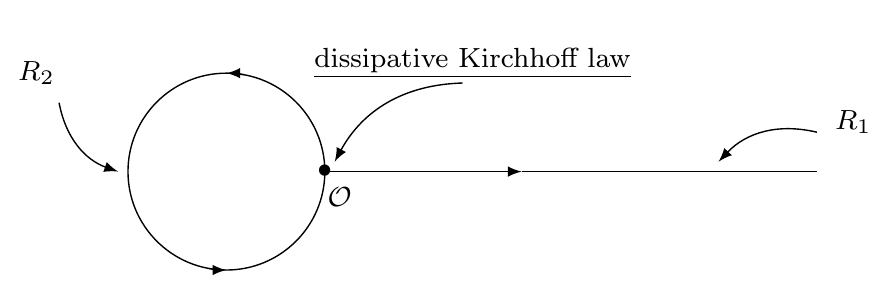}
\captionof{figure}{Tadpole graph}
\end{center}

Before a precise statement of our main result,
let us introduce some notations which will be used throughout the
rest of the paper.

\medskip

Let $R_i,i=1,2,$ be two disjoint sets
identified with a closed path of measure equal to $L > 0$ for $R_2$ and to $(0,+\infty),$
for $R_1$, see Figure \ref{fig1}. We set ${\mathcal R} := \ds \cup_{k=1}^2
\overline{R}_k$. We denote by $f = (f_k)_{k=1,2} =
(f_1,f_2)$ the functions on ${\mathcal R}$ taking their values in
$\CC$ and let $f_k$ be the restriction of $f$ to $R_k$.

\medskip

Define the Hilbert space 
$$
\mathcal{H} := \ds \oplus_{k=1}^2 L^2(R_k) = L^2 (\mathcal{R})
$$
with inner product
$$((u_k),(v_k))_{\mathcal H} = \ds \sum_{k=1}^2
(u_k,v_k)_{L^2(R_k)}$$
and introduce the following transmission
conditions (see \cite{kost,amamnic}):

\be
\label{t0}
(u_k)_{k=1,2} \in \ds \oplus_{k=1}^2
C(\overline{R_k}) \; \hbox{satisfies} \; u_1(0) = u_2(0)=u_2(L),
\ee

\be
\label{t1}
(u_k)_{k=1,2} \in \ds \oplus_{k=1}^2
C^1(\overline{R_k}) \; \hbox{satisfies} \; \ds \sum_{k=1}^2
\frac{du_k}{dx}(0^+) - \frac{du_2}{dx}(L^-)= i \, \alpha u_1 (0),
\ee
where $\alpha$ is a positive constant. 

\medskip

Let $H : {\mathcal D}(H) \subset {\mathcal H} \rightarrow {\mathcal H}$ be the linear operator on
${\mathcal H}$ defined by:
$$
{\mathcal D}(H) = \left\{(u_k)_{k=1,2} \in \oplus_{k=1}^2 H^2(R_k); \, (u_k)_{k=1,2} \;
\hbox{satisfies} \; \rfb{t0},\rfb{t1} \right\},
$$
$$
H (u_k)_{k=1,2} = (H_{k} u_k)_{k=1,2} = \left(- \frac{d^2 u_k}{dx^2} \right)_{k=1,2} = - \Delta_{{\mathcal R}}(u_k)_{k=1,2}.
$$

The operator $H$ generates a $C_0$ semigroup of contractions on $\mathcal{H}$, denoted by $(e^{itH})_{t \geq 0}$. We have in particular, for all $u_0 \in \mathcal{H}$, the following Schr\"odinger system
\begin{equation} 
\left\{
\begin{array}{l}
u^\prime = iH u, \, \mathcal{R} \times (0,+\infty),\\
u(0)= u_{0}, \, \mathcal{R},
\end{array}
\right.
\label{pbfirstorderch}
\end{equation}
admits a unique mild solution $u\in C([0,+\infty);\mathcal{H})$ and satisfies the energy identity:

\be
\label{energid}
\left\|u(t)\right\|^2_{\mathcal{H}} - \left\|u(0)\right\|^2_{\mathcal{H}} =- 2 \alpha \, \int_0^t \left|u_1(0,s)\right|^2 \, ds, \,\forall \, t  \geq 0. 
\ee

Here, we prove that the damped Schr\"odinger semigroup on the tadpole graph ${\mathcal R}$ satisfies an exponential stability
estimate. More precisely, we will prove the following theorem.
\begin{theorem} \label{mainresult1}
For all $t
\geq 0$, \be \label{dispest} \left\|e^{itH}
\right\|_{\mathcal{L}(\mathcal{H}_p^-)} \leq C \,
e^{-\omega t}, \ee where $C$ is a positive constant independent of $t$, $\omega= \frac{8\alpha}{3L}$ and $\mathcal{H}_p^-$ is a subspace of $\mathcal{H}$ which will be specified later (Theorem \ref{energy}).
\end{theorem}

The proof of this result is based on an appropriate decomposition of
the kernel of the resolvent that in particular gives a full characterization of the spectrum, made only of the
point spectrum and of the absolutely continuous one; showing the absence of a singular continuous part and to careful spectral analysis.

\medskip

The paper is organized as follows. The kernel of the resolvent is
given in Section \ref{sec1} (Theorem \ref{theo1}). Further all eigenfunctions of the dissipative Hamiltonian on the tadpole are constructed. They correspond to the confined modes on the head of the tadpole, which do not interact with the queue. The interaction is described by the absolutely continuous spectrum. Technically the main point is a decomposition
of the kernel of the resolvent in meromorphic and continuous terms (Theorem \ref{lsplitkernel}). The poles are shown to be the eigenvalues of the operator, the continuous term creates the absolutely continuous spectrum. The absence of further terms proves the absence of a singular continuous spectrum. Moreover, a careful spectral analysis (Theorem \ref{lemma1} and Proposition \ref{asyspec}) of the damped Schr\"odinger operator are given in Section \ref{sec1}. In Section \ref{basis}, we prove that the generalized eigenfunctions form a Riesz basis (in Theorem \ref{rieszbasis}) of some subspace of $\mathcal{H}$. The precise formulation and a proof of Theorem \ref{mainresult1} (Theorem \ref{energy}) is given in Section \ref{secenergy}. 

\section{The Resolvent and the spectrum} \label{sec1}
\subsection{The resolvent}
Given $z\in \{z\in \CC: \Im z>0, \, \Re z \leq 0\}$ and $g\in L^2({\mathcal R})$, we are looking
for $u \in {\mathcal D}(H)$ solution of
\[
- \Delta_{{\mathcal R}} u-z^2 u=g \hbox{ in } {\mathcal R}.
\]
Let us use the notation $\omega=-iz$.

Hence we look for $u$ in the form
$$
u_1(x) = \int_0^{+\infty} \frac{g_1(y)}{2\omega} \left( e^{-\omega |x-y|}-F_1(\omega) e^{-\omega (y+x)} \right) \,dy
$$
\be
\label{sergeu1}
- \int_0^L \frac{g_2(y)}{2\omega} \left(F_2(\omega) e^{-\omega (y+x)}+F_3(\omega) e^{\omega (y-x)} \right) \,dy,
\ee
$$
u_2(x) = \int_0^{+\infty} \frac{g_1(y)}{2\omega}\left(G_1(\omega) e^{-\omega (y+x)}+H_1(\omega) e^{-\omega (y-x)} \right) \,dy
$$
$$
 + \int_0^L \frac{g_2(y)}{2\omega} \left(e^{-\omega |x-y|}
+ G_2(\omega) e^{-\omega (y+x)}+G_3(\omega) e^{\omega (y-x)} \right.
$$
\be
\label{sergeu2}
\left.
+ \, H_2(\omega) e^{\omega (x-y)}+H_3(\omega) e^{\omega (x+y)} \right) \,dy,
\ee
where $F_i(\omega)$, $G_i(\omega)$ and  $H_i(\omega)$ , $i=1,2,3$ are constants calculated below
in order to satisfy  \rfb{t0}, \rfb{t1}. Indeed from these expansion, we clearly see that
for $k=1$ or 2:
\[
-u_k''+\omega^2 u_k= g_k \hbox{ in } R_k.
\]
Now we see that the continuity condition \rfb{t0} is satisfied if and only if
\[
\begin{array}{lll}
G_1+H_1=G_1e^{-\omega L}+H_1e^{\omega L}=1-F_1,\\
1+G_2+H_2=G_2e^{-\omega L}+H_2e^{\omega L}=-F_2,\\
G_3+H_3=(1+G_3)e^{-\omega L}+H_3e^{\omega L}=-F_3,
\end{array}
\]
while Kirchhoff condition
\rfb{t1} holds
 if and only if
\[
\begin{array}{lll}
(1 + \frac{i \alpha}{\omega})F_1+G_1(e^{-\omega L}-1)+H_1(1-e^{\omega L})=-1 + \frac{i \alpha}{\omega},\\
(1 + \frac{i \alpha}{\omega})F_2+G_2(e^{-\omega L}-1)+H_2(1-e^{\omega L})=-1,\\
(1 + \frac{i \alpha}{\omega})F_3+G_3(e^{-\omega L}-1)+H_3(1-e^{\omega L})=e^{-\omega L}.
\end{array}
\]

These equations correspond to three linear systems in $F_i, G_i, H_i$, $i=1,2,3$, whose associated matrix has a determinant $D(\omega)$ given by
\[
D_\alpha(\omega)=\left(1-\frac{i\alpha}{\omega}\right)e^{\omega L}\left(1- e^{-\omega L}\right)\left(e^{-\omega L}-\frac{3\omega+i\alpha}{\omega-i\alpha}\right).
\]
Since this determinant is different from zero (as $\Im z>0, \Re z \leq 0$), we deduce the following expressions:

\beq\label{sergef1}
F_1(\omega)=1+\frac{2\omega}{i\alpha-\omega}\frac{e^{-\omega L}+1}{e^{-\omega L}-\frac{3\omega+i\alpha}{\omega-i\alpha}},\\
\label{sergeg1}
G_1(\omega)=\frac{2\omega}{(i\alpha-\omega)\left(e^{-\omega L}-\frac{3\omega+i\alpha}{\omega-i\alpha}\right)},\\
\label{sergeh1}
H_1(\omega)=\frac{2\omega e^{-\omega L}}{(i\alpha-\omega)\left(e^{-\omega L}-\frac{3\omega+i\alpha}{\omega-i\alpha}\right)},\\
\label{sergef2}
F_2(\omega)=\frac{2\omega}{(i\alpha-\omega)\left(e^{-\omega L}-\frac{3\omega+i\alpha}{\omega-i\alpha}\right)},\\
\label{sergeg2}
G_2(\omega)=\frac{w+i\alpha}{\omega-i\alpha}\frac{1}{(e^{-\omega L}-1)\left(e^{-\omega L}-\frac{3\omega+i\alpha}{\omega-i\alpha}\right)}\\
\label{sergeh2}
H_2(\omega)=\frac{2\omega+(i\alpha-\omega)e^{-\omega L}}{\omega D_\alpha(\omega)},\\
\label{sergef3}
F_3(\omega)=-\frac{2 e^{-\omega L}  \left(e^{-\omega L}-1\right)}{D_\alpha(\omega)},\\
\label{sergeg3}
G_3(\omega)=\frac{(\omega+i\alpha)e^{-\omega L}}{\omega D_\alpha(\omega)},\\
\label{sergeh3}
H_3(\omega)=\frac{e^{-\omega L}}{D_\alpha(\omega)}\left( 2e^{-\omega L}-\frac{3\omega +i\alpha}{\omega}\right).
\eeq

Inserting these expressions in \rfb{sergeu1}-\rfb{sergeu2}, we have obtained the next result.

\begin{theorem} \label{theo1}
Let   $f \in \mathcal H$. Then, for $x \in {\mathcal R}$ and $z\in \mathbb{C}$ such that
$\Im z> 0, \,\Re z \leq 0$, we have
\be\label{resolvantformula} [R(z^2, H)f](x)= \int_{{\mathcal R}} K(x, x', z^2) f(x') \; dx',
\ee
where
the kernel $K$ is defined as follows:
\beq
\label{kernel11} K(x,y,z^2) &= &
    \frac{1}{2iz}  \left(e^{iz|x-y|} -F_1(-iz) e^{iz (x+y)}\right),   \forall x, y \in  {R_1},
\\
K(x,y,z^2) &=& -\frac{1}{2iz} \left(F_2(-iz) e^{iz (y+x)}+F_3(-iz) e^{-iz (y-x)}\right),  \forall x\in R_1, y\in R_2,
\label{kernel12}
\\
\label{kernel22}
K(x,y,z^2) &=&       \frac{1}{2iz}  \Big(e^{iz |x-y|}
+G_2(-iz) e^{iz (y+x)}+G_3(-iz) e^{-iz (y-x)}
\\
\nonumber
&+&H_2(-iz) e^{-iz (x-y)}+H_3(-iz) e^{-iz (x+y)} \Big),  \forall x, y \in  {R_2},
\\
K(x,y,z^2)&=&  \frac{1}{2iz} \left(G_1(-iz) e^{iz (y+x)}+H_1(-iz) e^{iz(y-x)} \right),  \forall x\in R_2, y\in R_1.
\label{kernel21}
\eeq
\end{theorem}

\subsection{The spectrum}
As usual, to obtain the resolution of the identity of $H$, we want to use the limiting absorption principle
that consists to pass to the limit in $K(x,y,z^2)$ as $\Im z$ goes to zero.
But in view of the presence of the factor $e^{iz L}-1$  in the denominator of $G_2, G_3, H_2, H_3$,
this limit is a priori not allowed. This factor comes from the circle $R_2$ and suggests that the real point spectrum is distributed in the whole
continuous spectrum. This is indeed the case has the next results will show.

\begin{theorem} \label{lemma1}
The spectrum of $H$ is given by the  disjoint union of $[0,+\infty[$ and $\Sigma_p^-(H)$ where
\begin{equation}\label{thespectrum}
 \Sigma_p^-(H)=\left\{\lambda^2; \lambda \in \mathbb{C}_+^+\; \text{and}\; e^{i\lambda L}=3-\frac{4\alpha}{\lambda+\alpha} \right\},
\end{equation}
where $\mathbb{C}_+^+:=\left\{z \in \mathbb{C};\, \Re z > 0, \Im z > 0 \right\}.$

\medskip

Indeed, for all $k\in \NN^*$, the number $\lambda_{2k}^2=\frac{4 k^2\pi^2}{L^2}$ is an eigenvalue of $H$
 and an associated eigenvector $\varphi^{(2k)}\in {\mathcal D}(H)$ given by
 \beq\label{vp1}
 \varphi^{(2k)}_1=0 \hbox{ in } R_1,\\
  \varphi^{(2k)}_2(x)=\frac{\sqrt{2}}{\sqrt{L}}\sin (\lambda_{2k} x), \forall x\in R_2.
  \label{vp2}
  \eeq
  Furthermore $\Sigma_p^-(H)$ is a discrete set of eigenvalues of $H.$ 
	\end{theorem}
\begin{proof}
We start by proving that the semi-axis $[0,+\infty[$ is a part of the spectrum of $H$ and that a  number $\lambda ^2\in [0,+\infty[$ is an eigenvalue if and only if there exists $k\in\NN^*$ such that $\lambda^2=\frac{4 k^2\pi^2}{L^2}.$

\medskip

 Let $\lambda ^2\ge 0$ and $\chi : \mathbb{R}\rightarrow [0,1]$ a smooth function such that $\chi(x)=1$ for $|x|\le 1/2$ and $\chi(x)=0$ for $|x|\ge 1.$  For $n\in\NN^*$ we define the ansatz $\theta_{n,\lambda}$ by $\theta_{n,\lambda}(x)=\frac{e^{i|\lambda|x}}{\sqrt{n}}\chi\left( \frac{x}{n}-1\right),$ for $x\ge 0.$ We indeed define the sequence
 $(\varphi^{(n)})_n=((\theta_n^\lambda,0))_n$.  We claim that $\lambda^2$ is an approximate eigenvalue of $H$ in the sense that $\varphi^{(n)}\in {\mathcal D}(H), \forall n\in\NN^*,  $ $\|\varphi^{(n)}\|_{\mathcal{H}}\not\to 0$   and $\|(H+\lambda^2)\varphi^{(n)}\|_{\mathcal{H}}\longrightarrow 0$  when $n\to +\infty.$ Then the Weyl criterion  proves that $\lambda^2$ is in the spectrum of $H.$
It is not hard to check that $\varphi^{(n)}\in\mathcal{D}(H)$ for all $n\ge 1$ and that 
$$
\|\varphi^{(n)}\|_{\mathcal{H}}^2=\|\theta^{\lambda}_n\|_{L^2}^2\ge\frac{1}{n}\int_{n/2}^{3n/2}dx=1. 
$$
Moreover, we remark that $\theta^{\lambda}_n(x)=e^{i|\lambda|x}\theta^{0}_n(x)$ for $n\ge 1$ and $x\ge 0$ and  
$$
\|(\theta^{0}_n)'\|_{L^2}^2\le \frac{2}{n^2}\|\chi'\|_{\infty}^2, \quad
\|(\theta^{0}_n)''\|_{L^2}^2\le \frac{2}{n^4}\|\chi''\|_{\infty}^2.
$$
Hence we obtain 
\begin{eqnarray}
\|(H+\lambda^2)\varphi^{(n)}\|_{\mathcal{H}}&=& \|(\theta^{\lambda}_n)''+\lambda^2\theta_n^\lambda\|_{L^2}\\
&\le& 2|\lambda|\|(\theta^{0}_n)'\|_{L^2}+\|(\theta^{0}_n)''\|_{L^2}\\
&\le& \frac{2}{n^2}\left( \sqrt{2\lambda^2n^2}\|\chi'\|_{\infty}+\|\chi''\|_{\infty}\right).
\end{eqnarray}
To prove that the numbers $\lambda_{2k}^2, k\in\NN^* $ are eigenvalues is direct since we readily check that $\varphi^{(2k)}$ defined by (\ref{vp1})-(\ref{vp2}) is indeed in ${\mathcal D}(H)$
and satisfies
$H\varphi^{(2k)}=\lambda_{2k}^2\varphi^{(2k)}$.

For the second assertion, we simply remark that if
$\varphi$ is an eigenvector of $H$ of eigenvalue $\lambda^2$, then
for $\lambda>0$, we have
\[
\varphi_1(x)= c_1 \sin (\lambda x)+c_2\cos (\lambda x), \forall x\in R_1,
\]
with $c_i\in \CC$. But the requirement that $\varphi_1$ belongs to $L^2(R_1)$ directly implies that
$c_1=c_2=0$. Hence $\varphi$ has to be in the form of the first assertion.
In the case $\lambda=0$,  $\varphi_1$ has to be zero and therefore
\[
\varphi_2(x)= c_1 +c_2x, \forall x\in R_2,
\]
with $c_i\in \CC$. By the continuity property at 0, we get
$c_1=c_1+c_2L=0$, hence $c_1=c_2=0$.\\
Now we look for other eigenvalues. Let $\lambda \in \mathbb{C}_+^+$ and let us find a solution of $Hu=\lambda^2 u$ in the domain ${\mathcal D}(H).$
The parameter $\lambda$ can be considered as a complex branch of the square root of $\lambda^2.$ Thus we will often choose the branch $\lambda$ such that $\Re(i\lambda^2)<0.$ This is because of the dissipativity of $iH.$

\medskip

A fundamental system of solution of the differential equation  $-\Delta f=\lambda^2 f$ is $\{e^{i\lambda x},e^{-i\lambda x}\}.$
Hence if we let $u_k(x)=A_ke^{i\lambda x}+B_ke^{-i\lambda x}, A_k\in \mathbb{C}, B_k\in\mathbb{C}$ for $k=1,2$ then $u=(u_1,u_2)\in \mathcal{D}(H)$ yields 
\begin{itemize}
	\item the continuity condition 
$$
A_1+B_1=A_2+B_2=A_2e^{i\lambda L}+B_2e^{-i\lambda L}
$$ 
\item and  the Kirchhoff transmission condition  
$$
\left( 1-\frac{\alpha}{\lambda}\right)A_1-\left( 1+\frac{\alpha}{\lambda}\right)B_1+\left( 1-e^{i\lambda L}\right)A_2-\left( 1-e^{-i\lambda L}\right)B_2=0.
$$
\end{itemize}
Since the eigenvector $(u_1,u_2)$ has to be in the domain of $H$ and $\Im \lambda \ge 0, $ we can see that $B_1=0.$ Therefore $A_1=A_2+B_2$ and $A_2,B_2$  are the solutions  of the linear system 
\begin{equation}
\left\{
\begin{array}{l}
(1-e^{i\lambda l})A_2+(1-e^{-i\lambda L})B_2=0\\
\left(2-\frac{\alpha}{\lambda}-e^{i\lambda L} \right)A_2+\left(e^{-i\lambda L}-\frac{\alpha}{\lambda} \right)B_2=0.
	\end{array}
\right.
\label{sys1}
\end{equation}
The determinant of the system is 
$$
d(\lambda):=e^{-i\lambda L}\left( \left(\frac{\alpha}{\lambda}+1 \right)e^{2i\lambda L}-4e^{i\lambda L}+3-\frac{\alpha}{\lambda}\right).
$$ 
The reals $\lambda_{2k}, k\in \NN^*$ are roots of $d(\lambda)=0$ and they correspond to the eigenvalues given in the first assertion.

Now if we set $T=e^{i\lambda L}$ then $d(\lambda)=0$ if and only if $(\frac{\alpha}{\lambda}+1)T^2-4T+3-\frac{\alpha}{\lambda}=0.$ The two roots of the previous polynomial are $T_1=1$ and $T_2=3-\ds\frac{4\alpha}{\lambda+\alpha}.$ Thus $\lambda$ is a zero of the  function $h_\alpha(\lambda)=e^{i\lambda L}+\ds\frac{4\alpha}{\lambda+\alpha}-3$ in the domain $\mathbb{C}_+^+.$ Since this is an open subset of $\mathbb{C}$ and $h_\alpha$ is analytic on it, then $h_\alpha$ has a discrete set of zeros with no accumulation point.
\end{proof}
\begin{remark}
\begin{enumerate}
	\item  We shall see below that the eigenvalues
$(\lambda_{2k}^2)_{k\in \mathbb{N}^*}$ are embedded in
the continuous spectrum with corresponding eigenfunctions
$\varphi^{(2k)}$ which are confined in the circle.
\item For the non real  eigenvalues $\lambda_{\alpha, k}^2$ we can see that there exists $c>0$ such that $\Im \lambda_{\alpha, k}^2\ge c, \forall k\in\NN.$ For such eigenvalues, we already know that $\lambda_{\alpha, k}$ satisfies  the relation 
\be
e^{i\lambda_{\alpha, k} L}=3-\ds\frac{4\alpha}{\lambda_{\alpha, k}+\alpha}.
\label{eqcar}
\ee
 Setting $\lambda_{\alpha,k}=\sigma_k+i\mu_k$ then $\Im\lambda_{\alpha,k}^2=2\sigma_k\mu_k. $  If we assume that there exists a subsequence of $\lambda_k$ such that $\Im\lambda_k^2\to 0$ when $k\to \infty$ then by taking the modulus in the relation (\ref{eqcar}) and  passing to the limit when $k\to \infty$ we get $1=3.$ This is a contradiction.
\item Using the two remarks above, the spectrum of $H$ is the union of two separated sets $\sigma_1$ and $\sigma_2$ where $\Sigma_p^-=\{\lambda_{\alpha,k}^2, k\in\NN\}$ and $\sigma_2=[0,+\infty[=\sigma_{c}(H)\cup\{\lambda_{2k}^2; k\in\NN^*\}$ and $\sigma_{c}$ is the  continuous spectrum of $H$ and contains the numbers $\lambda_{2k}^2, k\in\NN^*$ as a sequence of embedded eigenvalues.
 
\end{enumerate}
\end{remark}
\subsection{Asymptotic of the point spectrum}
Let $\alpha>0$ and recall that we have the following disjoint decomposition of the point spectrum  $$\Sigma_p(H)=\Sigma_p^+(H)\cup \Sigma_p^-(H)$$ with 
$
\Sigma_p^-(H)
$
is given by (\ref{thespectrum}) and 
$$
\Sigma_p^+(H)=\left\{\frac{4\pi^2k^2}{L^2}; k\in \NN^*\right\}.
$$
We want to find a precise localization of $\Sigma_p^-(H).$ We consider the meromorphic function  $h_\alpha$ defined by $$h_\alpha(\lambda)=e^{i\lambda L}+\frac{4\alpha}{\lambda +\alpha}-3.$$
We already know that $\lambda^2$ is an eigenvalue of $H$ if $h_\alpha(\lambda)=0$ and $\lambda\in \CC^+_+.$
The function $h_\alpha$ is analytic on $\CC^+_+$ so it has a discrete set of roots. Let's denote by $(\lambda_n(\alpha))_n$ the sequence of these roots.
When $\alpha=0,$ it is a straightforward  calculus to find that the numbers $\lambda_n(0)=\frac{2n\pi}{L}+i\frac{\ln(3)}{L}, n\in\mathbb{N}$ are roots of $h_0.$  Hence we have 
$$
\lambda_n^2(0)=\frac{4n^2\pi^2-\ln^2(3)}{L^2}+\frac{4in\pi \ln(3)}{L^2}.
$$
Now, we consider the operator family $\left\{H; \alpha\in\RR\right\}:=(H_\alpha)_{\alpha\in\RR}.$ This is an analytic family of type $B$ in the sense of Kato (\cite{kato}, chap. VII.4). By perturbation arguments, there exists, for small $\alpha,$  an analytic function $\alpha\mapsto \lambda_n^2(\alpha)$ such that $\lambda_n^2(\alpha)$ is in the spectrum of $H_\alpha.$ Let us write 
$$
\lambda_n(\alpha)=\lambda_n(0)+a\alpha+b\alpha ^2+O(\alpha^3)
$$
 where $a$ and $b$ are complex numbers we are going to compute.\\
On one hand, we have, for small $\alpha,$ 
\begin{eqnarray}
e^{iL\lambda_n(\alpha)}&=&e^{i\lambda_n(0)L}.e^{iLa\alpha+iLb\alpha^2+O(\alpha^3)}\nonumber\\
&=& 3\left(1+iaL\alpha+ibL\alpha^2-\frac{a^2L^2\alpha^2}{2}+O(\alpha^3)\right)\nonumber\\
&=& 3+3ia L\alpha+3\left(ibL-\frac{a^2L^2}{2}\right)\alpha^2+O(\alpha^3)\label{eq1}.
\end{eqnarray}
On the other hand, we can develop the following
\begin{eqnarray}
3-\frac{4\alpha}{\lambda_n(\alpha)+\alpha}&=& 3-\frac{4\alpha}{\lambda_n(0)+(a+1)\alpha+b\alpha^2+O(\alpha)^3}\nonumber\\
&=& 3-\frac{4\alpha}{\lambda_n(0)}+\frac{4(a+1)}{\lambda_n^2(0)}\alpha^2+O(\alpha^3).\label{eq2}
\end{eqnarray}
Using (\ref{eqcar}) and identifying (\ref{eq1}) and (\ref{eq2}), yield
$$
a=\frac{4i}{3L\lambda_n(0)}$$
and 
$$ b=\frac{16}{9L^2\lambda_n^3(0)}-\frac{4i}{9L\lambda_n^2(0)}.
$$
Now, we want to prove that the asymptotic expansion valid for small $\alpha$ remains true for large $n$ and any fixed $\alpha.$ The $\lambda_n(\alpha)$ satisfies the relation $(\alpha+\lambda_n)e^{i\lambda_n L}=3\lambda_n-\alpha.$ Deriving this equation with respect to $\alpha$ yields
$$
\lambda_n'(\alpha)=\frac{4i\lambda_n(\alpha)}{3L\lambda_n^2(\alpha)+2\alpha L\lambda_n(\alpha)-\alpha^2 L+4i\alpha}.
$$
Using the fact that $|\lambda_n(\alpha)|$ behaves like $\frac{2n\pi}{L}$ for large $n,$ we can see that $|\lambda'_n(\alpha)|=O\left(\frac{1}{n}\right).$ Thus 
\begin{equation}\label{asymtlambda}
\lambda_n(\alpha)=\frac{2n\pi}{L}+i\frac{\ln(3)}{L}+O(\frac{1}{n}), \; n\to +\infty.
\end{equation}
Then we obtain the  asymptotic expansion of $\lambda_n(\alpha)$ for large $n$ in the same way we did for small $\alpha.$ 
In conclusion, we have the following result.
\begin{proposition} \label{asyspec}
Let $\alpha>0$ fixed. Then when $n$ is large enough, the eigenvalue $\lambda_n^2(\alpha)$ has the following expansion:
\begin{equation}\label{asymtlambdacarre}
\lambda^2_n(\alpha)=\frac{4\pi^2}{L^2}n^2-\frac{\ln^2(3)}{L^2}+i\left(\frac{8\alpha}{3L}+\frac{4\pi \ln(3)}{L^2}n\right)+O\left(\frac{1}{n}\right).
\end{equation}
\end{proposition}
At this stage we define the projection $P_{pp}$   on the closed subspace spanned by the $\varphi^{(2k)}$'s, namely
for any $f\in {\mathcal H}$, we set
\[
P^+_{pp} f=\sum_{k=0}^{+\infty} (f, \varphi^{(2k)})_{{\mathcal H}} \varphi^{(2k)}.
\]

Note that $P^+_{pp} f$ is different from $f$ on $R_2$
because $L^2(R_2)$ is spanned by the set of eigenvectors of the Laplace operator with Dirichlet boundary conditions at $0$ and $L$, that are the set $\{\varphi^{(\ell)}_2\}_{\ell\in \NN^*}$,
where
\[
\varphi^{(\ell)}_2(x)=\frac{\sqrt{2}}{\sqrt{L}}\sin (\lambda_{\ell} x), \forall x\in R_2,
\]
and $\lambda_{\ell} =\frac{\ell\pi}{L}$. Hence
\[
f- P^+_{pp} f=\sum_{k=0}^{+\infty} \left(\int_0^Lf(x)\varphi^{(2k+1)}_2(x)\,dx\right) \varphi^{(2k+1)}_2.
 \]

We show, according to \cite[Corollary 2.6]{amamnic}, that our operator has no singular continuous 
spectrum and $P_{c}f = f -( P_{pp}^++ P_{pp}^-)f, \, \forall\, f \in \mathcal{H}.$

\begin{theorem}\label{lsplitkernel}
For all $z\in \{z\in \CC: \Im z>0, \, \Re z \leq 0\}$, and all $x,y\in  {\mathcal
R}$, the kernel $K(x,y,z^2)$ defined in Theorem \ref{theo1} admits
the decomposition
\be\label{splitkernel} K(x,y,z^2)=K_c(x,y,z^2)+K_{pp}^+(x,y,z^2)+K_{pp}^-(x,y,z^2), \ee
where for $x,y  \in R_2$ and $X = e^{izL}$ we have
$$
  K_c(x,y,z^2) =
	$$
	$$
	-\frac{\sin zy\sin zx}{2iz}+\frac{1}{2iz}\left(\frac{i\alpha+\omega}{i\alpha-\omega}e^{-izy}e^{izx}+\frac{2\omega}{i\alpha-\omega}\left(X -\frac{2(i\alpha+\omega)}{i\alpha-\omega}\right)e^{-izy}e^{-izx}\right),
$$
$$
K_{pp}^-(x,y,z^2) = 
$$
$$
\frac{ iX\sin(zy)e^{-izx}-\sin(zy)\sin zx-\frac{2(i\alpha+\omega)}{(i\alpha-\omega)^2}e^{-izy}e^{izx}-\frac{\omega A(\alpha)}{i\alpha-\omega}e^{-izy}e^{-izx}}{iz(X-\omega_c)}
$$
and
\begin{equation*}
K_{pp}^+(x,y,z^2) = -\frac{\cos\left(\frac{zL}{2}\right)}{2z\sin\left(\frac{zL}{2}\right)}\sin(zy)\sin zx.
\end{equation*}
The function $z  \mapsto K_c(x,y,z^2)$ is continuous on $\Im
z\geq 0, \Re z \leq 0$ except at $z=0$.\\
 The function while $z \mapsto K_{pp}^+(x,y,z^2)$ is
meromorphic in $\CC$ with poles at the points $\lambda_{2k}$, $k\in
\NN^*$ and at $z=0$,\\
The function while $z \mapsto K_{pp}^-(x,y,z^2)$ is
meromorphic in $\CC$ with poles at the points $\lambda_{\alpha, k}$, $k\in
\NN.$ .

For $x \not\in R_2$ or $y \not\in R_2$
we have $K_p(x,y,z^2) = 0$ and $K_c(x,y,z^2)=K(x,y,z^2)$ as defined
in Theorem \ref{theo1}.
\end{theorem}

\begin{proof}
Let us recall that  
$$
D_\alpha(w)=\left( \frac{\omega-i\alpha}{\omega}\right)e^{\omega L}(1-e^{-\omega L})(e^{-\omega L}-w_c)
$$
with $\omega =-iz$ and $\omega_c=\frac{3\omega+i\alpha}{\omega-i\alpha}.$
 The problem in the decomposition of resolvent kernel $K(x,y,z^2)$  only appears  for $x$ and $y$ in $R_2$, since in the other cases,
$K$ has no poles and therefore in that cases we simply take $K_{pp}^+=K_{pp}^-=0$.
Hence we need to perform this splitting for $x,y\in R_2$ using the formula (\ref{kernel22}).
On one hand we obtain
$$
G_2(\omega)e^{izy}+G_3(\omega)e^{-izy}=\left(\frac{i\alpha+\omega}{i\alpha-\omega}\right)\left(\frac{X+1}{
X-\omega_c}e^{-izy}-\frac{2i}{(X-1)(X-\omega_c)}\sin zy \right).
$$ 
On the other hand we get
$$
H_2(\omega)e^{izy}+H_3(\omega)e^{-izy} =
$$
$$
\frac{1}{X-\omega_c}\left( Xe^{izy}+\frac{i\alpha+\omega}{i\alpha-\omega}\frac{e^{izy}}{X-1}+\frac{2\omega}{i\alpha-\omega}X^2e^{-izy}
-\frac{\omega+i\alpha}{i\alpha-\omega}\frac{X^2}{X-1}e^{-izy} \right)
$$
$$
=\left(\frac{i\alpha+\omega}{i\alpha-\omega}\right)\frac{2i\sin zy}{(X-1)(X-\omega_c)}+\frac{2iX\sin zy}{X-\omega_c}+\frac{2\omega}{i\alpha-\omega}\left(\frac{X^2-X-\frac{i\alpha+\omega}{2\omega}}{X-\omega_c} \right)e^{-izy}.
$$
Now an elementary calculus gives
$$
\frac{X^2-X-\frac{i\alpha+\omega}{2\omega}}{X-\omega_c}=X+\frac{2(\omega+i\alpha)}{\omega-i\alpha}-A(\alpha)\frac{1}{X-\omega_c},
$$
$$
\left(\frac{i\alpha+\omega}{i\alpha-\omega}\right)\frac{2i}{(X-1)(X-\omega_c)}=i\left( \frac{1}{X-1}-\frac{1}{X-\omega_c}\right),
$$
$$
\left(\frac{i\alpha+\omega}{i\alpha-\omega}\right)\frac{X+1}{X-\omega_c}=\frac{i\alpha+\omega}{i\alpha-\omega}-\frac{4\omega(i\alpha+\omega)}{(\omega-i\alpha)^2}\frac{1}{X-\omega_c}.
$$
and
$$
\frac{2iX}{X-\omega_c}=2i+\frac{2i\omega_c}{X-\omega_c},
$$
where we have set 
$$
A(\alpha)=\frac{(\omega+i\alpha)(11\omega^2+\alpha^2+6i\alpha\omega)}{2\omega(\omega-i\alpha)^2}.
$$
Thus, 
\begin{eqnarray*}
Q:=(G_2(\omega)e^{izy}+G_3(\omega)e^{-izy})e^{izx}+(H_2(\omega)e^{izy}+H_3(\omega)e^{-izy})e^{-izx}=\\
\frac{i\alpha+\omega}{i\alpha-\omega}\frac{4\sin zy\sin zx}{(X-1)(X-\omega_c)} +\left(\frac{i\alpha+\omega}{i\alpha-\omega}-\frac{4\omega(i\alpha+\omega)}{(\omega-i\alpha)^2}\frac{1}{X-\omega_c}\right)e^{-izy}e^{izx}\\
+\left(2i+\frac{2i\omega_c}{X-\omega_c}\right)\sin zy e^{-izx} +\left( \frac{2\omega}{i\alpha-\omega}X-\frac{4\omega(\omega+i\alpha)}{(i\alpha-\omega)^2}\right)e^{-izy} e^{-izx}\\
+\frac{2\omega A(\alpha)}{i\alpha-\omega}\frac{1}{X-\omega_c}e^{-izy} e^{-izx}.
\end{eqnarray*}
We decompose the last expression into two terms
$$
K_0(x,y,z):=\frac{i\alpha+\omega}{i\alpha-\omega}\frac{4\sin zy\sin zx}{(X-1)(X-\omega_c)}
$$
and
\begin{eqnarray*}
K_1(x,y,z):=\left(\frac{i\alpha+\omega}{i\alpha-\omega}\right)\left( 1-\frac{4\omega}{i\alpha -\omega}\frac{1}{X-\omega_c}\right)e^{-izy} e^{izx}\\
+\frac{2iX}{X-\omega_c}\sin zy e^{-izx}+\frac{2\omega}{i\alpha-\omega}\left(X-\frac{A(\alpha)}{X-\omega_c} -\frac{2(i\alpha+\omega}{i\alpha-\omega}\right)e^{-izy} e^{-izx}.
\end{eqnarray*}
We go further in the calculus and we find the following splitting of $Q$
$$
K_0(x,y,z)=-i\frac{\cos\left(\frac{zL}{2}\right)}{\sin\left(\frac{zL}{2}\right)}\sin zy\sin zx-\left(1+\frac{2}{X-\omega_c} \right)\sin zy\sin zx.
$$
Here we used the fact that $X=e^{izL}$ and thus 
$$\frac{1}{X-1}=-\frac{1}{2}+\frac{\cos\left(\frac{zL}{2}\right)}{2i\sin\left(\frac{zL}{2}\right)}.
$$ 
Going back to the formula for $Q$ and assembling together the terms with $\displaystyle{\frac{1}{X-\omega_c}}$, then  $Q$ is the sum of three terms 
$$
Q=Q_c+Q_{pp}^++Q_{pp}^-
$$
with 
$$
Q_c=-\sin zy\sin zx+\frac{i\alpha+\omega}{i\alpha-\omega}e^{-izy}e^{izx}+\frac{2\omega}{i\alpha-\omega}\left(X -\frac{2(i\alpha+\omega)}{i\alpha-\omega}\right)e^{-izy}e^{-izx},
$$
$$
Q_{pp}^-=\frac{ 2iX\sin(zy)e^{-izx}-2\sin(zy)\sin zx-\frac{4(i\alpha+\omega)}{(i\alpha-\omega)^2}e^{-izy}e^{izx}-\frac{2\omega A(\alpha)}{i\alpha-\omega}e^{-izy}e^{-izx}}{X-\omega_c},
$$
and
$$
Q_{pp}^+=-i\frac{\cos\left(\frac{zL}{2}\right)}{\sin\left(\frac{zL}{2}\right)}\sin(zy)\sin zx.
$$
\end{proof}

\section{Riesz basis} \label{basis}

\noindent In this section, it is proved that the generalized eigenfunctions of the dissipative operator $iH$ associated to the eigenvalues in $i\Sigma_p^- (H)$ form a Riesz basis of the subspace of $\mathcal{H}$ which they span (denoted by $\mathcal{H}_p^-$). 
To this end, we recall that a sequence $(\Psi_n)_{n\in\NN}$ is a Riesz basis in a Hilbert space $V$ if there exist a Hilbert space $\tilde V,$ an orthonormal basis $(e_n)_{n\in\NN}$ of $\tilde{V}$ and an isomorphism $\Theta : \tilde{V}\longrightarrow V$ such that $\Theta e_n=\Psi_n, \forall  n\in \NN.$

\medskip 
 
We start by computing the eigenfunctions of $iH$ corresponding to the eigenvalues $\lambda_n^2(\alpha).$  So let $n$ and $\alpha $ be fixed. The eigenfunction associated to $\lambda_n^2(\alpha)$ is of the form $\psi^-_n=(\psi^-_{1,n},\psi^-_{2,n})$ with 
$$
\psi^-_{1,n}(x)=A_{1,n}e^{i\lambda_n(\alpha)x}, \quad \text{for}\; x\in[0,+\infty[$$
and
$$ \psi^-_{2,n}=A_{2,n}e^{i\lambda_n(\alpha)x}+ B_{2,n}e^{-i\lambda_n(\alpha)x} , \quad \text{for}\; x\in[0,L]$$
  where $ A_{1,n}, A_{2,n}, B_{2,n}$ are complex constants. 
Since $\psi^-_n$ has to be in the domain of $H$, we obtain the following
\begin{equation}
\psi^-_{1,n}(x)=C_n\frac{4\lambda_n(\alpha)}{\lambda_n(\alpha)+\alpha}e^{i\lambda_n(\alpha)x}
\label{fonctpro1}
\end{equation}
and
\begin{equation}
\psi^-_{2,n}(x)=C_n\left( e^{i\lambda_n(\alpha)x}+\frac{3\lambda_n(\alpha)}{\lambda_n(\alpha)+\alpha}e^{-i\lambda_n(\alpha)x}\right)
\label{fonctpro2}
\end{equation}
where $C_n$ is a constant such that $\|\psi^-_n\|_{L^2}=1.$
\begin{theorem} \label{rieszbasis}[Riesz basis for the operator $iH$] 
The generalized eigenfunctions of $iH$ forms a Riesz basis of $\mathcal{H}_p^-.$
\end{theorem}
We recall that $\lambda_n(\alpha)$ and $\lambda_n^2(\alpha)$ have the asymptotic expansions given in (\ref{asymtlambda}) and (\ref{asymtlambdacarre}). Thus one can prove the following behavior of the family $\left(\psi^-_{2,n} \right)_n.$
\begin{lemma}
There exists a constant $C>0$ such that for all $n, m\in\NN $ with $n<m$ we have
\begin{equation}
<\psi^-_{2,n},\psi^-_{2,m}>_{L^2(0,L)}\leq \frac{C}{<n>(m-n)}, \quad where \; <n>=(1+n^2)^{1/2}.
\label{lemma}
\end{equation}
\end{lemma}
\begin{proof}
We consider the functions $\phi_n(x)=e^{i\lambda(\alpha)x}$ and $\tilde{\phi}_n(x)=e^{-i\lambda(\alpha)x}, \; x\in[0,L].$
Using  (\ref{asymtlambda}) we obtain 
$$
\|\phi_n\|_{L^2(0,L)}^2=\|\tilde{\phi}_n\|_{L^2(0,L)}^2=\frac{4L}{\ln(3)}+O \left(\frac{1}{n} \right), \; 
$$
and
$$
<\phi_n,\tilde{\phi}_n>_{L^2(0,L)}=O \left(\frac{1}{n} \right).
$$
Indeed, for $n\neq m,$ we obtain
$$
<\phi_n,\phi_m>_{L^2(0,L)}=\frac{e^{i(\lambda_n-\bar\lambda_m)L}-1}{\lambda_n-\bar\lambda_m}
$$
and
$$
<\tilde{\phi}_n,\tilde{\phi}_m>_{L^2(0,L)}=\frac{e^{-i(\lambda_n-\bar\lambda_m)L}-1}{\bar\lambda_m-\lambda_n}.
$$
We know that $\lambda_n-\bar\lambda_m=(n-m)\nu+O(\frac{1}{n})$ for $n$ large, with $_nu=\frac{2\pi}{L}$. Hence 
$$
|<\phi_n,\phi_m>|+|<\tilde{\phi}_n,\tilde{\phi}_m>|\le \frac{1}{m-n}, 
$$
and 
$$
|<\phi_n,\phi_m>+<\tilde{\phi}_n,\tilde{\phi}_m>|\le \frac{1}{<n>(m-n)}.
$$
By the same arguments we get 
$$
|<\phi_n,\tilde\phi_m>|+|<\tilde{\phi}_n,{\phi}_m>|\le \frac{1}{m+n}, 
$$
and 
$$
|<\phi_n,\tilde\phi_m>+<\tilde{\phi}_n,{\phi}_m>|\le \frac{1}{<n>(m+n)}.
$$
Going back to $\psi_{2,.}^-$ we write $\psi_{2,n}^-=C_n\left( \phi_n+\frac{3\lambda_n(\alpha)}{\lambda_n(\alpha)+\alpha}\tilde\phi_n\right)$ and we obtain
\begin{eqnarray*}
\left|<\psi^-_{2,n},\psi^-_{2,m}>\right|&=&\left|C_n\bar{C}_m\left(<\phi_n,\phi_m>+ <\phi_n,\frac{3\lambda_m(\alpha)}{\lambda_m(\alpha)+\alpha}\tilde\phi_m>\right.\right.\\
&+&\left.\left.\frac{3\lambda_n(\alpha)}{\lambda_n(\alpha)+\alpha}<\tilde\phi_n,\phi_m>+\frac{3\lambda_n(\alpha)}{\lambda_n(\alpha)+\alpha}\frac{3\bar\lambda_m(\alpha)}{\bar\lambda_m(\alpha)+\alpha}<\tilde{\phi}_n,\tilde{\phi}_m>\right)\right|.
\label{eq:}
\end{eqnarray*}
We conclude using the asymptotic $|\lambda_n(\alpha)|=\frac{2n\pi}{L}+O(\frac{1}{n})$ for $n$ large enough.
\end{proof}

\begin{proof}[Proof of Theorem \ref{rieszbasis}] 
Consider now the map $\Theta : \ell^2(\CC)\rightarrow L^2(0,+\infty)\times L^2(0,L)$ given by 
$$
\Theta\left( (a_n)_{n\in\NN}\right)=\sum_{n=0}^{+\infty}a_n\Psi^-_n.
$$
We claim that $\Theta$ is an isomorphism from $\ell^2(\CC)$ to $V_p:=\overline{Sp(\Psi^-_n)}$ where $\Psi^-_n=(\psi^-_{1,n},\psi^-_{2,n}).$
Using the relations (\ref{fonctpro1}) and (\ref{fonctpro2}), it is enough to consider the functions $\phi_n$ and $\tilde{\phi}_n.$
Let $a=(a_n)_{n\in\NN}\in \ell^2(\CC).$ For $N, p$ fixed in $\NN$ we have 
\begin{eqnarray}
\left\| \sum_{n=N}^{N+p}a_n\phi_n\right\|^2 -\sum_{n=N}^{N+p}|a_n|^2 &=& \sum_{j=N}^{N+p}\sum_{k=j+1}^{N+p}2 \Re\left( a_j\bar{a_k}<\phi_j,\phi_k>\right)\\
&\leq& 2 C\sum_{j=0}^{+\infty}\frac{|a_j|}{<j>}\sum_{k=1}^{+\infty}\frac{|a_{k+j}|}{k}\nonumber\\
&\leq & \tilde{C}\|a\|^2_{\ell^2(\CC)}.
\label{eq:bis}
\end{eqnarray}
where $C$ is the constant given in (\ref{lemma}) and $\tilde{C}>0.$ This proves that the series $ \ds \sum_{n=0}^{+\infty}a_n\phi_n$ converges in $L^2(0,L)$ as $a=(a_n)_n\in \ell^2(\CC).$ Moreover, letting $p\rightarrow +\infty$ and taking $N=0$ give
$$
\left\| \sum_{n=0}^{+\infty}a_n\phi_n\right\|^2\leq \tilde{C}\|a\|^2_{\ell^2(\CC)}.
$$
Thus $\Theta $ is continuous. It remains to prove that $\Theta$ is injective. For this, let $a\in\ell^2(\CC)$ and assume that $ \ds \sum_{n=0}^{+\infty}a_n\phi_n=0$ in $L^2(0,L).$
Let $ n\ne m\in \NN.$ First, we have
\begin{eqnarray*}
\lambda_n^2(\alpha)<\phi_n,\tilde{\phi}_m>= <H\phi_n,\tilde{\phi}_m>\\
&= &<\phi_n,H^*\tilde{\phi}_m>\\
&=&<\phi_n,\lambda_m^2(\alpha)\tilde{\phi}_m>\\
&=&\lambda_m^2(\alpha)<\phi_n,\tilde{\phi}_m>.
\end{eqnarray*}
This proves that $<\phi_n,\tilde{\phi}_m>=0$ for $n\neq m.$ 
Next, we can see from (\ref{fonctpro1}, \ref{fonctpro2}) that $<\phi_n,\tilde{\phi_n}>\neq 0$ for $n> N$ with $N$ large enough. Hence, taking the inner product of $ \ds \sum_{n=0}^{+\infty}a_n\phi_n=0$ by $\tilde{\phi_p}$ with $p$ fixed yields $0= \ds \sum_{n=0}^Na_n \phi_n=0$ in $L^2(0,L).$ Applying the operator $H$ iteratively $N$ times gives the following sequence of relations:
$$
0=\sum_{n=0}^N \lambda^{2k}_n(\alpha)a_n\phi_n=0, \quad \text{in}\; L^2(0,L), \; \text{for}\; k=0,\dots, N.
$$
This is a Vendermonde problem. Since the $\lambda_n^{2k}$ are pairwise different, we deduce that $a_n\phi_n=0$ and hence $a_n=0$ for all $n.$ 
\end{proof}

\section{Energy decreasing} \label{secenergy}

Using the Riesz basis constructed in the latest section, the energy is proved to decrease exponentially to a non-vanishing value depending on the initial datum. The decay rate is explicitly given at the end of Theorem \ref{energy} below since the $\lambda$'s satisfying $i \lambda^2 \in i \, \Sigma_p^- (H)$. The decay rate is computable numerically. 

\subsection{Energy decreasing using the Riesz basis}

\begin{theorem} \label{energy}[Energy decreasing of the solution]
Let $E(t) := \frac{1}{2} \, \left\|u(t)\right\|^2_{\mathcal{H}}$ be the energy, $\mathcal{H}_p^-$ (respectively $\mathcal{H}^+_p$) be the subspace of $\mathcal{H}$ spanned by the $\psi^-(\lambda, \cdot)$'s (resp. $\psi^+(\omega, \cdot)$'s), which are the normalized (in $\mathcal{H}$) eigenfunctions of $iH$ associated to the eigenvalues $i \lambda^2$ in $i \Sigma_p^- (H)$ (resp. $i \Sigma_p^+ (H)$). 

\begin{enumerate}
\item $\mathcal{H}_p^-$ is orthogonal to $\mathcal{H}^+_p$.  
\item Let $u_0$ in $\mathcal{H}_p^- \oplus \mathcal{H}_p^+$ be the initial condition of the boundary value problem given in the introduction and $u_0^+$ its orthogonal projection onto $\mathcal{H}_p^+$. \\
Then $E(t)$ decreases exponentially to $E^+(0):= \frac{1}{2} \, \|u_0^+ \|_{\mathcal{H}}^2$ when $t$ tends to $+ \infty$. More precisely

\be \label{optimaldecay}
E(t) = E^+ (t) + E^-(t) = E^+(0) + E^-(t) \leq E^+(0) + e^{-2\omega t} E^-(0)
\ee 
where 
$$-\omega:= \sup_{ \{ i \lambda^2 \in i\Sigma_p^- (H)\} } \Re(i \lambda^2) < 0.$$ 
\end{enumerate}
\end{theorem}

\begin{proof} 
First part: Above all, it is easy to see that the operator $(iH)^*$ is obtained by changing $i$ by $-i$ in $iH.$ Thus, if $iH\psi^+(\lambda,.)=i\lambda^2 \psi^+(\lambda,.),$ then $(iH)^* \psi^+(\lambda,.) =- i \lambda^2 \psi^+(\lambda,.).$ 

\medskip

Now, to prove that $\mathcal{H}^+_p$ is orthogonal to $\mathcal{H}^-_p$, it suffices to check that any generalized 
eigenfunction $\psi^-(\lambda',.)$ of $\mathcal{H}_p^-$ is orthogonal to any eigenfunction $\psi^+(\lambda,.)$ of $\mathcal{H}_p^+.$ 

\medskip

First we assume that $\psi^-_p(\lambda',.)$ is an eigenfunction, i.e $$iH \psi^-(\lambda',.)= i\lambda'^2 \psi^-(\lambda',.).$$ Therefore, since $i \lambda^2$ is purely imaginary,  

$$\begin{array}{lll}
i \lambda'^2 <\psi^-(\lambda',.),\psi^+(\lambda,.)>_{\mathcal{H}} &=& <iH \psi^-(\lambda',.),\psi^+(\lambda,.)>_{\mathcal{H}}\\
&=&< \psi^-(\lambda',.), (iH)^* \psi^+(\lambda,.)>_{\mathcal{H}}\\
&=&-< \psi^-(\omega',.), iH \psi^+(\lambda,.)>_{\mathcal{H}}\\
&=&-\overline{i\lambda^2}<\psi^-(\lambda',.),\psi^+(\omega,.)>_{\mathcal{H}}\\
&=&i\lambda^2 <\psi^-(\lambda',.),\psi^+(\lambda,.)>_{\mathcal{H}}.
\end{array}
$$

Consequently $<\psi^-(\lambda',.),\psi^+(\lambda,.)>_{\mathcal{H}}=0.$

Secondly, we assume that $\lambda'$ is not simple. Let $\psi^2(\omega',.)$ be an associated generalized eigenfunction of order $p\geq 2,$ in the sense that

$$(iH-i\lambda'^2)^p \psi^-(\lambda',.)=0,\;(iH-i\lambda'^2)^{p-1} \psi^-(\lambda',.)\neq 0.$$

Setting $\psi=(iH-i\lambda'^2) \psi^-(\lambda',.),$ then $\psi$ is a generalized eigenfunction associated to $i\lambda'^2$ of order $p-1,$ so arguing by iteration with respect to the order $p$ we can assume that $<\psi,\psi^+(\lambda,.)>_{\mathcal{H}}=0.$ 

\medskip

Therefore
$$
\begin{array}{lll}
i\lambda'^2 <\psi^-(\lambda',.),\psi^+(\lambda,.)>_{\mathcal{H}}&=& <iH\psi^-(\lambda',.)+\psi,\psi^+(\lambda,.)>_{\mathcal{H}}  \\
&=& <iH\psi^-(\lambda',.),\psi^+(\lambda,.)>_{\mathcal{H}}\\
&=&i\lambda^2 <\psi^-(\lambda',.),\psi^+(\lambda,.)>_{\mathcal{H}},
\end{array}
$$
as previously. Consequently $<\psi^-(\lambda',.),\psi^+(\lambda,.)>_{\mathcal{H}}=0.$ 

\medskip

\noindent Second part: the purely point spectrum of the dissipative operator $iH$ is the union of $i \Sigma^+_p (H)$ set of the purely imaginary eigenvalues and $i\Sigma_p^- (H)$ set of the other eigenvalues. 

\medskip

The initial condition $u_0$ is written as a sum of two terms: 

$$
u_0 := \sum_{i \lambda^2 \in i\Sigma_p^+(H)} u^+_0(\lambda, \cdot) \psi^+(\lambda, \cdot) + \sum_{i \lambda^2 \in \Sigma_p^- (H)} u^-_0(\lambda, \cdot) \psi^-(\lambda, \cdot)
$$

where $\psi^+(\lambda, \cdot)$ (respectively $\psi^-(\lambda, \cdot)$) is a normalized (in $\mathcal{H}$) eigenfunction of $iH$ associated to the eigenvalue $i \lambda^2$ in $i\Sigma_p^+(H)$ (resp. $i\Sigma_p^-(H)$). Note that the sum takes into account the multiplicities of the eigenvalues here. 

\medskip

Thus the solution of the boundary value problem given in the introduction is:

$$
u(t) := \sum_{i \lambda^2 \in i\Sigma_p^+ (H)} u^+_0(\lambda, \cdot)  e^{i \lambda^2 t} \psi^+(\lambda, \cdot) + \sum_{i \lambda^2 \in 
i\Sigma_p^- (H)} u^-_0(\lambda, \cdot) e^{i \lambda^2 t} \psi^-(\lambda, \cdot).
$$    

The energy $E(t)= E^+(t) + E^-(t)$ with 

$$E^+(t):= 1/2 \, \sum_{i \lambda^2 \in i\Sigma_p^+(H)} \left \| u^+_0(\lambda, \cdot) \right \|_{\mathcal{H}}^2  |e^{i  \lambda^2 t}|^2, 
$$
$$
E^-(t):= 1/2 \, \sum_{i \lambda^2 \in i \Sigma_p^- (H)} \left \| u^-_0(\lambda, \cdot) \right \|_{\mathcal{H}}^2 |e^{i \lambda^2 t}|^2.
$$
Now, since $i \Sigma_p^+ (H)$ contains only purely imaginary eigenvalues, $|e^{2i \lambda^2 t}| = 1$, for any $\lambda$ such that $i \lambda^2 \in i \Sigma_p^+ (H)$ and any $t > 0$. Thus $E^+(t)=E^+(0)$ for any $t > 0$. 

\medskip

The real part of $i \lambda^2$ is a non-positive real number if $\lambda$ is such that $i \lambda^2 \in i\Sigma_p^- (H)$. This real part is proved to be equal to $- \frac{8\alpha}{3L}$ if $i \lambda^2 \in i\Sigma_p^- (H)$.   

\medskip

It holds $E^-(t) \leq e^{-2\omega t} E^-(0)$. Thus $E^-(t)$ decreases exponentially to $0$ when $t$ tends to $+ \infty$ and the total energy $E(t)$ decreases exponentially to $E^+(0)$ when $t$ tends to $+ \infty$. 
\end{proof}

\subsection{Energy decreasing: a numerical example}

The purely point spectrum of the conservative operator $iH$, with $\alpha =0$, is given by $i\Sigma_p^+ (H)$.

Then the set $i\Sigma_p^- (H)$ which is a part of the spectrum of the dissipative operator $iH$ has a vertical asymptote:  

$$
\Re(i\lambda^2)= - \frac{8\alpha}{3L},
$$

\begin{figure}[h]
\begin{center}
\includegraphics[width=12cm, height=8cm]{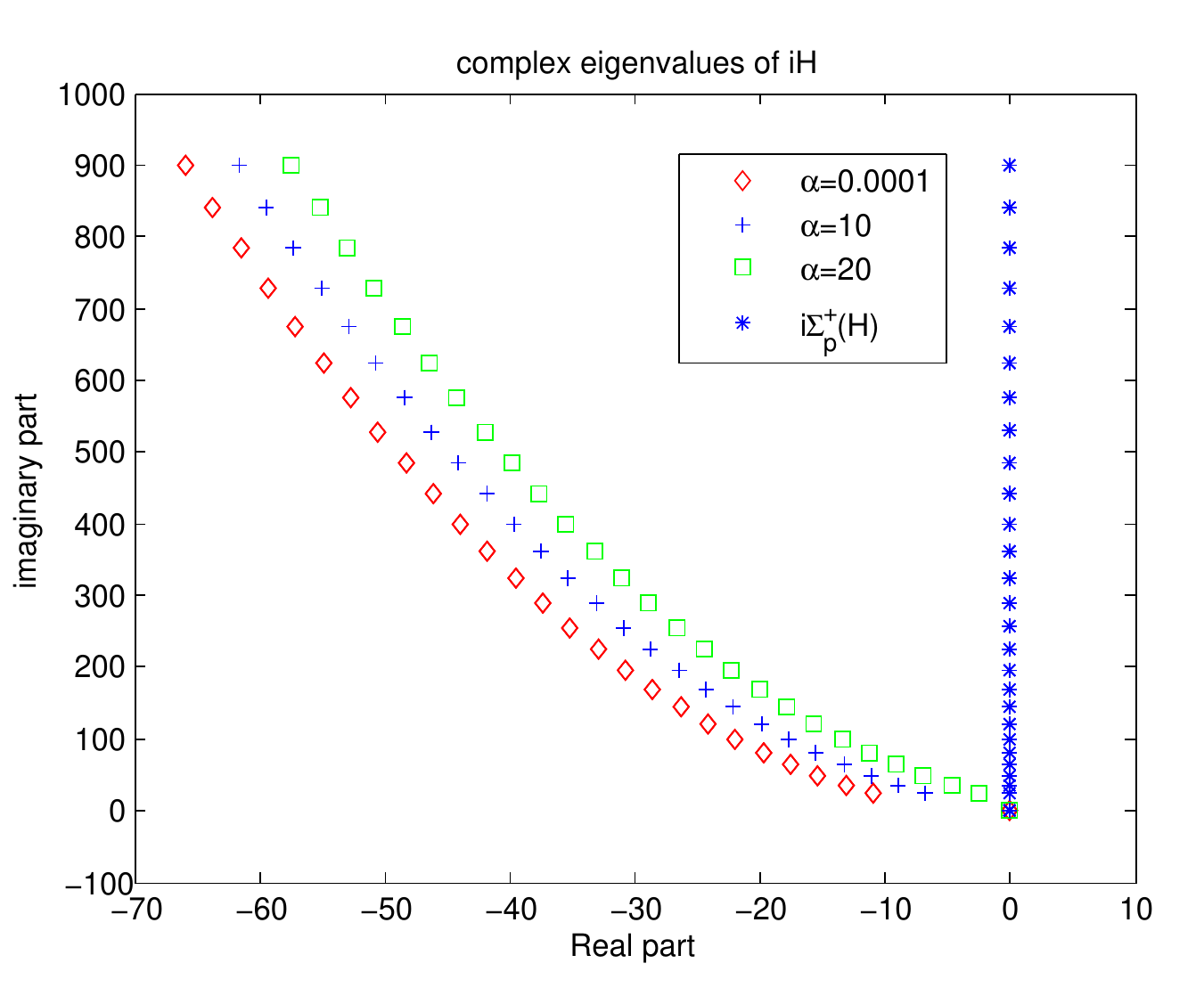}%
\caption{$\lambda_n(\alpha), $\; for $n=1,\dots, 30$ and $L=2\pi$}%
\label{fig2}%
\end{center}
\end{figure}

which is consistent with the numerical computation of the spectrum, see Figure \ref{fig2}.


\begin{thebibliography}{99}

\bibitem{amamnic} F. Ali Mehmeti, K. Ammari and S. Nicaise, Dispersive effects for the Schr\"odinger equation on the tadpole graph, {\em Journal of Mathematical Analysis and Applications,} {\bf 448} (2017), 262--280.

\bibitem{amregmer} K. Ammari, D. Mercier and V. R\'egnier, Spectral analysis of the Schr\"odinger operator on binary tree-shaped networks and applications, {\em J. Differential Equations,} {\bf 259} (2015), 6923--6959.

\bibitem{ammarinicaise} K. Ammari and S. Nicaise, {\em Stabilization of elastic systems by collocated feedback}, Lecture Notes in Mathematics, 2124, Springer, Cham, 2015.

\bibitem{cacc} C. Cacciapuoti, D. Finco and D. Noja, Topology-induced bifurcations for the nonlinear Schr\"odinger equation on the tadpole graph, {\em Phys. Rev.} E 91, 013206, 2015.

\bibitem{kato} T. Kato, {\em Perturbation theory for linear operators,} Classics in Mathematics, Springer Verlag, 1980.

\bibitem{kost} V. Kostrykin and R. Schrader, Kirchhoff's rule for quantum wires, {\em J. Phys. A: Math. Gen.,} {\bf 32}
(1999), 595--630.

\bibitem{nojaetall:15} D.  Noja,  D. Pelinovsky and G. Shaikhova,
Bifurcations and stability of standing waves in the nonlinear Schr\"odinger equation on the tadpole graph,
{\em  Nonlinearity,} {\bf 28} (2015), 2343--2378.

\end{thebibliography}
\end{document}